\documentclass[a4, 11pt] {amsart} 
\usepackage{times}
\usepackage{geometry}
\usepackage{amssymb}                                                                 					
\usepackage{amsthm}
\usepackage{amsrefs}  
\usepackage{mathrsfs}  								    					
\usepackage{verbatim}

\usepackage{tikz}

\usetikzlibrary{matrix}													
\usetikzlibrary{shapes}													

\usepackage[cmtip,arrow]{xy}												

\usepackage{tikz-cd}									    					

\usepackage{graphicx}								  					 
\usepackage{float}									    					 
\usepackage{epstopdf}
\usepackage{pgfplots}
\pgfplotsset{compat=1.12}

\usepackage{braids}
	
\usepackage{subcaption}	
\usetikzlibrary{snakes}	
\usepackage{tikz-cd}			    
\tikzset{snake arrow/.style=
{-,
decorate,
decoration={snake,amplitude=.2mm,segment length=1mm,post length=0mm}}
}

\usepackage{enumerate}								     					
\usepackage{thmtools}
\usepackage{thm-restate}
\usepackage{hyperref}
\usepackage[capitalise]{cleveref}
\usepackage{mathtools}
\usepackage{bbm}									     					
\usepackage{framed}	
\usepackage{eufrak}

\newtheorem{thm}{Theorem}

\newtheorem{cor}{Corollary}
\newtheorem{lem}{Lemma}
\newtheorem{prop}{Proposition}
\theoremstyle{definition}
\newtheorem{defn}{Definition}

\newtheorem{exm}{Example}
\newtheorem{rem}{Remark}

\newcommand{\Hom}{\text{Hom}}

\renewcommand{\Im}{\text{Im}}

\newcommand{\cA}{\mathcal{A}}

\newcommand{\cK}{\mathcal{K}}

\newcommand{\cV}{\mathcal{V}}

\newcommand{\B}{\mathbb{B}}
\newcommand{\C}{\mathbb{C}}

\newcommand{\K}{\mathbb{K}}

\newcommand{\R}{\mathbb{R}}

\newcommand{\Z}{\mathbb{Z}}

\newcommand{\lra}{\longrightarrow}

\DeclareRobustCommand\longtwoheadrightarrow
     {\relbar\joinrel\twoheadrightarrow}

\newcommand{\f}[5]{
\begin{array}{rcl}
#1:#2 & \longrightarrow & #3 \\
#4 & \longmapsto & #5  \\
\end{array}
}

\DeclareMathOperator{\Spec}{\operatorname{Spec}}
\DeclareMathOperator{\codim}{codim}	

\title{Higher Order Degrees of Affine Plane Curve Complements}
\author{Eva Elduque}
\address{Department of Mathematics,
University of Wisconsin-Madison,
480 Lincoln Drive, 
Madison WI 53706-1388, USA}
\email
{evaelduque@math.wisc.edu}

\author{Laurentiu Maxim}
\address{Department of Mathematics,
University of Wisconsin-Madison,
480 Lincoln Drive, 
Madison WI 53706-1388, USA}
\email
{maxim@math.wisc.edu}

\date{\today}

\keywords{plane curve complement, essential line arrangement, singularities, higher order degrees, Alexander polynomial}

\subjclass[2010]{32S25, 32S55, 32S05, 32S20, 57M27}

\begin{document}

\maketitle

\begin{abstract} 
We study finiteness (and vanishing) properties of the higher order degrees associated to complements of complex affine plane curves with mild singularities at infinity.
Our results impose new obstructions on the class of groups that can be realized as fundamental groups of affine plane curve complements. We also clarify the relationship between the higher order degrees and the multivariable Alexander polynomial of a  non-irreducible plane curve.
\end{abstract}


\section{Introduction}

In knot theory, a strategy to address problems that the
Alexander polynomial is not strong enough to solve is to consider
non-abelian invariants (e.g., see \cite{Cochran}). These are
Alexander-type invariants of coverings corresponding to terms of the
derived series of a knot group, and share most of the properties of
the classical Alexander invariants. Despite the
difficulties of working with modules over non-commutative rings,
there are applications to estimating knot genus, detecting fibered,
prime and alternating knots, and to knot concordance.
Higher order Alexander invariants can be associated to any finitely
presented group $G=\pi_1(X)$, in terms of coverings of $X$
given by the terms in the rational derived series of $G$. These
in turn have striking applications if one considers the fundamental
group of a link complement or that of a closed
$3$-manifold \cite{Harvey}. For example, they can be used to obtain lower bounds for
the Thurston norm, and provide new algebraic obstructions to a
$4$-manifold of the form $M^3 \times S^1$ admitting a symplectic
structure.

Motivated by their success in the classical knot theory and
low-dimensional topology,  C. Leidy and the second author initiated in
\cite{MaxLeidy} the study of higher order Alexander-type invariants for complex affine plane
curve complements. The exploration of topology of complex plane curves and of their complements is a subject that goes back to works of Zariski, Enriques, Hirzebruch, 
Deligne, or Fulton, and which has flourished in more recent research endeavors by Libgober, Dimca, Suciu, Artal-Bartolo, Cogolludo-Agust\'in, etc.
As the fundamental group of a plane curve complement is in general highly non-abelian, 
one typically considers invariants of the fundamental group that still capture most of the topology of the curve, but which are more manageable, e.g., Alexander-type invariants.  

To any affine plane curve $C \subset \C^2$, in \cite{MaxLeidy} one associates a sequence
$\{\delta_n(C)\}_n$ of (possibly infinite) integers, called the \emph{higher order
degrees} of $C$. Roughly speaking, these integers measure the
``sizes'' of quotients of successive terms in the rational derived
series $\{G_r ^{(n)}\}_{n \geq 0}$ of the
fundamental group $G=\pi_1(\C^2 \setminus C)$ of the curve complement (see Definition \ref{hodg}). It was also noted in
\cite{MaxLeidy} that the higher order degrees of plane curves (at any level
$n$) are sensitive to the ``position'' of singular points, this being one of the initial motivations  for adapting and studying  Alexander-type invariants in the context of plane curve complements.

While in theory higher order degrees of a plane curve complement can be computed by Fox free calculus from a presentation of  $G=\pi_1(\C^2 \setminus C)$, such calculations are in general tedious, see \cite{MaxLeidySurvey} for some examples. Furthermore, 
as these integers can also be interpreted as Betti-type invariants associated to the tower of coverings of $\C^2 \setminus C$ corresponding to the subgroups $G_r ^{(i)}$ (the first of which is the universal abelian cover), a priori there is no reason to expect that such invariants have any good vanishing or finiteness properties. 
The main result of \cite{MaxLeidy} proved that for curves in general position at infinity (i.e., whose projective completion is transversal in the stratified sense to the line at infinity in $\C P^2$) these higher order degrees are in fact \emph{finite}, and a uniform upper bound was given  only in terms of the degree of the curve. More precisely, one has the following result: 
\begin{thm}\cite[Corollary 4.8]{MaxLeidy}\label{th0i} \ If $C\subset \C^2$ is a reduced plane curve of degree $m$, in general position at infinity, then:
$$
\delta_n(C)\leq m(m-2), \text{\ \ for all }n\text{.}
$$
\end{thm}

\medskip

One of the goals of this paper is to provide generalizations of Theorem \ref{th0i} to various contexts in which the assumption of good behavior at infinity is relaxed. As the sequence of higher order
degrees  of a plane curve is an invariant of the fundamental group
of the complement, a better understanding of its properties (such as finiteness) will
impose new obstructions on the class of groups that can be realized as
fundamental groups of affine plane curve complements.

Our first result generalizes Theorem \ref{th0i} to the context of essential complex line arrangements. (Note that such line arrangements are not necessarily in general position at infinity.) We prove the following (see Theorem \ref{th2}):
\begin{thm}\label{th2i}  Assume that the complex affine plane curve $C$ defines an essential line arrangement  $\cA=\{L_1,\ldots, L_m\}\subset \C^2$  (that is, not all lines in $\cA$ are parallel).
Then,
$$
\delta_n(C)\leq m(m-2), \text{\ \ for all }n\text{.}
$$
Moreover, the equality holds only in the case where $\cA$ consists of $m$ lines going through one point, and in that case the equality holds for all $n\geq 0$.
\label{thmbound}
\end{thm}

In the special case when an arrangement  contains one line which meets all other lines transversally, we show in Theorem \ref{th4} the following result (which was asserted without a proof in \cite{MaxLeidySurvey}):
\begin{thm}\label{th4i}
Assume that the affine plane curve $C$ defines a line arrangement  $\cA=\{L_1,\ldots, L_m\}\subset \C^2$, which is obtained from an essential line arrangement
 $\cA'=\{L_1,\ldots, L_{m-1}\}$ by adjoining a line $L_m$ that is transversal to every line in $\cA'$ (that is, the singularities of the curve $C$ along the irreducible component $L_m$ consist of $m-1$ nodes). Then
$$
\delta_n(C)= 0, \text{ for all }n\geq 0.
$$
\end{thm}

At the opposite spectrum, i.e., if the plane curve $C$ defines a line arrangement $\cA=\{L_1,\ldots, L_m\}\subset \C^2$ consisting of $m$ distinct parallel lines, then an easy calculation shows that (see Proposition \ref{par}):
$$
\delta_n(C)=
\begin{cases}
\infty, & m>1, \\
0, & m=1,
\end{cases}
$$ 
for all $n \geq 0$.

We also prove the following generalization of Theorem \ref{th0i} in the context when the plane curve $C$ is allowed to have {mild} singularities at infinity. More precisely, we show the following (see Theorem \ref{th3}):
\begin{thm}\label{th3i}
Let $C\subset \C^2$ be a reduced plane curve of degree $m$, let $\overline C$ be its closure in $\C P^2$ and let $L_\infty$ be the line at infinity. Suppose that the intersections of $L_\infty$ and $\overline C$ are either transversal or $L_\infty$ is the tangent line to $\overline C$ at a smooth point and it is a simple tangent there (i.e., it has multiplicity $2$). If any of the following two conditions hold
\begin{enumerate}
\item[(a)] $m=2$;
\item[(b)] at least one of the intersections of $\overline C$ and $L_\infty$ is transversal,
\end{enumerate}
then 
$$
\delta_n(C)\leq m(m-2),
$$ for all $n \geq 0$.
\end{thm}
Moreover, we generalize \cref{th4i} to the context of plane curves as follows (see \cref{thmVanishing}):
\begin{thm}\label{th5i}
Let $n\geq 0$. Assume that the affine plane curve $C$ is of the form $C=L\cup C'$, where $C'$ is a curve of degree $m-1$ in $\C^2$ such that $\delta_n(C')$ is finite, and $L$ is a line transversal to $C'$ such that $L\cap C'$ consists of $m-1$ distinct points. Then,
$$
\delta_n(C)= 0.
$$
\end{thm}

A natural question in the context of non-commutative Alexander-type invariants is to relate the higher order degrees of a plane curve complement to the previously studied Alexander-type invariants, such as the Alexander polynomials. Preliminary steps in this direction have already been made in \cite{MaxLeidy}, were the authors showed
that if the (one-variable) Alexander polynomial of an irreducible plane curve is
trivial then {\it all} higher-order degrees $\{\delta_n\}_n$ vanish. (If the curve is irreducible, then $\delta_0(C)$ is the degree of the Alexander polynomial of $C$.)
In relation with the universal abelian
invariants of a curve, it was also noted in \cite{MaxLeidy} that if the codimension (in the character torus) 
of the first characteristic variety of the plane curve complement is $>1$ then
$\delta_0(C)=0$ (this fact was first pointed out by A. Libgober in an informal conversation with the second author, see also Corollary \ref{corcodim}). However, curves (e.g., in
general position at infinity) may have supports of codimension 
one in the character variety (cf. \cite{Li}), and for this boundary case we show here that  $\delta_0(C)$ is the degree
of the multivariable Alexander polynomial $\Delta_C$ of the plane curve $C$ (see Theorem \ref{thmAP}).


\medskip

The paper is structured as follows. In Section \ref{notations}, we recall the definition of higher order degrees of an affine plane curve complement. Section \ref{cla} supplies proofs of Theorems \ref{th2i} and \ref{th4i}. Theorems \ref{th3i} and \ref{th5i} are proved in Section \ref{singinf}. Finally, in Section \ref{char}, we indicate the relation between $\delta_0(C)$ and the degree of the multivariable Alexander polynomial of the plane curve complement in the case when $C$ is not irreducible.


\section{Higher-order invariants of a plane curve complement}\label{notations}

Though most of the background material presented in this section applies to any finitely presented group, we focus mainly on fundamental groups of complex affine plane curve complements.

Let $C=\{f(x,y)=0\}$ be a reduced curve
in $\C^2$ of degree $m$, with complement $$U:=\C^2\setminus C,$$ and denote by 
$G:=\pi_1(U)$ the fundamental group of its complement.
If $C$ has $s$ irreducible components, then \begin{equation}\label{eq1} H_1(G;\Z)=H_1(U;\Z)=G/G'=\Z^s,\end{equation} generated by meridian loops about the smooth parts of the irreducible components of
$C$.

In this section we recall the definition of the higher-order Alexander-type invariants of the group $G$. These were originally used in the study of knots and, respectively,
$3$-manifolds, see e.g.,  \cite{Cochran,Harvey}, and they were ported to the study of plane curve complements in \cite{MaxLeidy}, e.g., to show that certain groups cannot be realized
as fundamental groups of such complements. 

\begin{defn}\label{hodg} The {\it rational derived serie}s of the group $G$ is defined as follows:
$G_r ^{(0)}=G$, and for $n \geq 1$,
$$G_r ^{(n)}=\{g \in G_r ^{(n-1)} \mid g^k \in
[G_r ^{(n-1)},G_r ^{(n-1)}], \ \text{for some} \ k \in \Z \setminus \{0\}
\}.$$ 
\end{defn}

It is easy to see that $G_r ^{(i)} \triangleleft G_r ^{(j)}
\triangleleft G$, if $i \geq j \geq 0$. 
The successive quotients of the rational derived series are
torsion-free abelian groups. In fact (cf. \cite[Lemma 3.5]{Harvey}),
$$G_r^{(n)}/G_r^{(n+1)} \cong \left(G_r^{(n)}/[G_r^{(n)},
G_r^{(n)}] \right)/\{\Z-\text{torsion}\}.$$ Therefore, for 
$G=\pi_1(\C^2 \setminus C)$ we get from (\ref{eq1}) that $G'=G_r^{(1)}$.

The use of the rational derived
series as opposed to the usual derived series is needed in order to avoid
zero-divisors in the group ring $\Z\Gamma_n$, where $$\Gamma_n:=G/G_r ^{(n+1)}.$$
By construction,  $\Gamma_n$ is a
poly-torsion-free-abelian group, in short a PTFA (\cite[Corollary 3.6]{Harvey}), i.e., it
admits a normal series of subgroups such that each of the successive
quotients of the series is torsion-free abelian. Then $\Z\Gamma_n$
is a right and left Ore domain, so it embeds in its classical right
ring of quotients $\cK_n$, a skew-field. Every module over $\cK_n$ is a free module, and such modules have a well-defined rank $\text{rk}_{\cK_n}$ which is additive on short exact sequences. (These statements also apply to the right ring of quotients $\cK$ of the group ring $\Z \Gamma$ of any PTFA group $\Gamma$, e.g., see \cite[Remark 2.4]{MaxLeidy} and the references therein.)

\begin{defn}\rm The \emph{$n$-th order Alexander module} of (the complement of) the plane curve $C$
is defined as
$$\mathcal{A}^{\Z}_n(C)=H_1(U;\Z\Gamma_n)=H_1(U_{\Gamma_n};\Z),$$
where $U_{\Gamma_n}$ is the covering of $U$ corresponding to the
subgroup $G_r^{(n+1)}$.  That is,
$$\mathcal{A}^{\Z}_n(C)=G_r^{(n+1)}/[G_r^{(n+1)},G_r^{(n+1)}],$$ viewed as a
right $\Z \Gamma_n$-module.\newline The \emph{$n$-th order rank}
of (the complement of) $C$ is:
$$r_n(C)=\text{rk}_{\mathcal{K}_n}H_1(U;\cK_n).$$
\end{defn}

\begin{rem}\rm
Note that
$\mathcal{A}^{\Z}_0(C)=G_r^{(1)}/[G_r^{(1)},G_r^{(1)}]=G'/G''$ is
just the universal abelian Alexander module of the complement.
\end{rem}

\begin{exm}\label{sm-nod}  If the curve $C$ is in general position at infinity (i.e., the line at infinity in $\C P^2$ is transversal in stratified sense to the projective completion of $C$), and it is nonsingular or has only nodal singular points (i.e., locally defined by $x^2-y^2=0$), then $G=\pi_1(\C^2\setminus C)$ is abelian, and therefore
$\mathcal{A}^{\Z}_n(C)=0$ for all $n$ (e.g., see \cite[Remark
3.4]{MaxLeidy}).\end{exm}

In \cite{MaxLeidy}, one associates to any plane curve $C$ (or, equivalently, to the fundamental group $G$ of its complement) a
sequence of non-negative integers $\delta_n(C)$ as follows (it is
more convenient to work over a principal ideal domain, or a PID for short, so we look for a ``convenient"
one): Let $\psi \in H^1(G;\Z)$ be the primitive class representing
the linking number homomorphism $$G \overset{\psi}{\lra} \Z, \ \alpha
\mapsto \text{lk}(\alpha,C).$$ Since $G'$ is in the kernel of $\psi$,
we have a well-defined induced epimorphism $\bar{\psi} : \Gamma_n
\to \Z$. Let $\bar{\Gamma}_n =\ker \bar{\psi}$. Then $\bar{\Gamma}_n$
is a PTFA group, so $\Z\bar{\Gamma}_n$ has a right ring of quotients
$$\K_n=(\Z\bar{\Gamma}_n)S_n ^{-1},$$ where $S_n=\Z\bar{\Gamma}_n \setminus \{ 0 \}$.
Set $$R_n:=(\Z\Gamma_n)S_n^{-1}.$$ Then $R_n$ is a flat left
$\Z\Gamma_n$-module.

A very important role in what follows is played by the fact that $R_n$ is a PID; in fact, $R_n$ isomorphic to the ring of skew-Laurent
polynomials $\K_n[t^{\pm 1}]$. This can be seen as follows: by choosing a $t \in
\Gamma_n$ such that $\bar {\psi} (t)=1$, we get a splitting $\phi$
of $\bar{\psi}$, and the embedding $\Z\bar{\Gamma}_n \subset \K_n$
extends to an isomorphism $R_n \cong \K_n[t^{\pm 1}]$. However this
isomorphism depends in general on the choice of splitting of $\bar
\psi$.

\begin{defn}\rm
(1) \ The \emph{$n$-th order localized Alexander module} of the plane 
curve $C$ is defined to be $$\mathcal{A}_n(C)=H_1(U;R_n),$$ viewed
as a right $R_n$-module.  If we choose a splitting $\phi$ to
identify $R_n$ with $\K_n[t^{\pm 1}]$, we define
$\mathcal{A}^{\phi}_n(C)=H_1(U;\K_n[t^{\pm 1}])$. \newline (2) \
The \emph{$n$-th order degree of $C$} is defined to be:
$$\delta_n(C)=\text{rk}_{\K_n} \mathcal{A}_n(C)=\text{rk}_{\K_n} \mathcal{A}^{\phi}_n(C).$$
\end{defn}

\begin{rem}\rm Note that $\delta_n(C) < \infty$ if and only if
$\text{rk}_{\mathcal{K}_n}H_1(U;\mathcal{K}_n)=0$, i.e.
$\mathcal{A}_n(C)$ is a torsion $R_n$-module.
\end{rem}

\begin{rem}\rm
If the plane curve $C$ is irreducible, then $\delta_{0}(C)$ is the degree of the Alexander polynomial of $C$; see \cite[Remark 3.9]{MaxLeidy}.
\end{rem}

The higher order degrees $\delta_n(C)$ are integral invariants of the fundamental
group $G$ of the complement (endowed with the linking number homomorphism). Indeed, by \cite{Harvey}, one has:
\begin{equation}\label{eq2} \delta_n(C)=\text{rk}_{\K_n} \left(
{G^{(n+1)}_r}/[G^{(n+1)}_r,G^{(n+1)}_r] \otimes_{\Z\bar{\Gamma}_n}
\K_n \right).\end{equation} Note that since the isomorphism between $R_n$ and $\K_n[t^{\pm
1}]$ depends on the choice of splitting, one cannot define in
a canonical way a higher-order version of the Alexander polynomial. However, for any choice of splitting, the degree of the associated higher-order Alexander
polynomial is the same, hence this yields a
well-defined invariant of $G$, which is exactly the higher-order
degree $\delta_n$ defined above.

The higher-order degrees of $C$ may be computed by means of Fox free calculus from a presentation of $G=\pi_1(\C^2 \setminus C)$, see \cite[Section 6]{Harvey} for details. Such computational techniques will be used freely in this paper.

It was shown in \cite{MaxLeidy} that if $C$ is an irreducible plane curve, or a curve in general position at infinity (i.e., for which the line at infinity in $\C P^2$ is transversal in the stratified sense to the projective completion of $C$), then the higher-order degrees  $\delta_n(C)$ are finite. More precisely, one has the following:
\begin{thm}\label{th1} If $C\subset \C^2$ is a reduced plane curve of degree $m$, in general position at infinity, then:
$$
\delta_n(C)\leq m(m-2), \text{\ \ for all }n\text{.}
$$
In particular, the $n$-th order Alexander module  $\mathcal{A}^{\Z}_n(C)$ is a torsion
$\Z\Gamma_n$-module, for all $n$.
\end{thm}

One of the goals of this paper is to provide generalizations of Theorem \ref{th1} to various contexts in which the assumption of good behavior at infinity is relaxed.


\section{Complex line arrangements}\label{cla}

Our first result, Theorem \ref{th2} below, generalizes Theorem \ref{th1} to the context of essential complex line arrangements. In Theorem \ref{th4} we study the special class of arrangements containing a line with only nodal singularities.

Assume that all irreducible components of the reduced plane curve $C$ are complex lines, i.e., the defining polynomial $f=\prod_{i=1}^m \ell_i$ factorizes into a product of linear forms $\ell_i:\C^2 \to \C$, $i=1,\ldots,m$. Let $$L_i:=\ker (\ell_i),$$ and let $$\cA:=\{L_1,\ldots, L_m\}\subset \C^2$$ be the corresponding complex line arrangement, with complement $U$. As before, we will use the notation $\delta_n(C)$ for the higher-order degrees of the complement $U:=\C^2 \setminus C=\C^2 \setminus \cA$.


\subsection{Upper bounds on higher-order degrees}\label{clab}
In this section, we prove the following generalization of Theorem \ref{th1} to the context of essential complex line arrangements.
\begin{thm}\label{th2}  Assume that the complex affine plane curve $C$ defines an essential line arrangement  $\cA=\{L_1,\ldots, L_m\}\subset \C^2$  (that is, not all lines in $\cA$ are parallel).
Then,
$$
\delta_n(C)\leq m(m-2), \text{\ \ for all }n\text{.}
$$
Moreover, the equality holds only in the case where $\cA$ consists of $m$ lines going through one point, and in that case the equality holds for all $n\geq 0$.
\label{thmbound}
\end{thm}

At the opposite spectrum (i.e., if the essentiality assumption is dropped), we have the following:
\begin{prop}\label{par} If the plane curve $C$ defines a line arrangement $\cA=\{L_1,\ldots, L_m\}\subset \C^2$ consisting of $m$ distinct parallel lines, then:
$$
\delta_n(C)=
\begin{cases}
\infty, & m>1, \\
0, & m=1,
\end{cases}
$$ 
for all $n \geq 0$.
\end{prop}

\begin{proof}
In this case, $\C^2\setminus C$ is homotopy equivalent to a wedge sum of $m$ circles. If $m=1$, we have that $\pi_1(\C^2\setminus C)\cong \Z$, so it is abelian. It then follows from (\ref{eq2}) that
$$
\delta_n(C)=0, \text{ for all }n\geq 0.
$$

Suppose now that $m>1$. The chain complex computing $H_*(\C^2\setminus C;R_n)$ looks like
$$
\cdots\rightarrow 0 \rightarrow (R_n)^m \rightarrow  R_n\rightarrow 0
$$
Hence, $H_1(\C^2\setminus C;R_n)$ is a non-zero free right $R_n$-module, so
$$
\delta_n(C)=\infty, \text{ for all } n\geq 0.
$$
\end{proof}

Theorem \ref{th2} is a consequence of the following two preparatory lemmas (Lemma \ref{l1} and Lemma \ref{l2}). In Lemma \ref{l1} we consider the case when there is a line in the arrangement which has no singularities at infinity, whereas in Lemma \ref{l2} every line is assumed to have singularities at infinity.

\begin{lem}\label{l1}
In the notations of Theorem \ref{th2}, assume that there exists a line in $\cA$ such that no other line in $\cA$ is parallel to it. Then,
$$
\delta_n(C)\leq m(m-2), \text{\ \ for all }n\text{.}
$$
Moreover, the equality is achieved only in the case when $C$ consists of $m$ lines going through one point, and in that case the equality holds for all $n\geq 0$.
\end{lem}

\begin{proof}
Reordering, we can assume that $L_1$ is not parallel to any other line in $\cA$. Let $P_1,\ldots,P_r$ be the singular points of $C$ in $L_1$. Let $$F=L_1\setminus \bigsqcup_{i=1}^r (L_1\cap \B_i^4)$$ be the (real) surface obtained by removing small balls $\B_i^4\subset \C^2$ around the singular points $P_i$. Hence $F$ is obtained from $L_1$ by removing a $2$-dimensional open disk $D_i$ around every singular point $P_i$.

Let $$N=F\times S^1.$$ Here $N$ should be thought of as the boundary of a tubular neighborhood around the non-singular part of $L_1$. We have that $\partial N=\partial F\times S^1$, and since $\partial F$ is a union of disjoint $S^1$'s (one from every disk $D_i$ removed), then $\partial N$ is a union of disjoint tori $\bigsqcup\limits_{i=1}^r T_i$ (again, one from every disk $D_i$ removed). Let us fix a point $Q_i$ in the circle $S^1$ corresponding to the boundary of the disk $D_i$ for every such disk removed.

Let $d_i$ be the number of lines in $\cA$ going through the singular point $P_i$, $i=1,\ldots,r$. Let $K_i$ be the link of the singularity at the point $P_i$ (hence $K_i$ is a Hopf link with $d_i$ components), and let $S^3_i$ be the boundary of $\B_i^4$. We consider the space
$$
X=N\cup_{\left(\bigsqcup\limits_{i=1}^r T_i\right)} \left(\bigsqcup_{i=1}^r S^3_i\setminus K_i\right)\subset \C^2\setminus C
$$
where the gluing is done as follows: A meridian around the component of $K_i$ corresponding to the line $L_1$ is glued to $\{Q_i\}\times S^1\subset N$, and the component of $K_i$ corresponding to $L_1$ is glued to the $S^1$ corresponding to the boundary of $D_i$.

The homology of the space $X$ with $R_n$-coefficients can be computed from the corresponding Mayer-Vietoris sequence:
\begin{equation}
\begin{array}{c}
\cdots \rightarrow H_2(N;R_n)\oplus \left(\bigoplus\limits_{i=1}^r H_2(S^3_i\backslash K_i;R_n)\right) \rightarrow H_2(X;R_n)\xrightarrow{\alpha} \\
\xrightarrow{\alpha}  \bigoplus\limits_{i=1}^r H_1(T_i;R_n)\rightarrow H_1(N;R_n)\oplus \left(\bigoplus\limits_{i=1}^r H_1(S^3_i\backslash K_i;R_n)\right)\rightarrow H_1(X;R_n)\rightarrow\\
\rightarrow \bigoplus\limits_{i=1}^r H_0(T_i;R_n)\rightarrow H_0(N;R_n)\oplus \left(\bigoplus\limits_{i=1}^r H_0(S^3_i\backslash K_i;R_n)\right)\rightarrow H_0(X;R_n)\rightarrow 0
\end{array}
\label{eqnMVseq}
\end{equation}

Hence, using the additivity of the rank of $\K_n$-modules (recall that $\K_n$ is the right ring of quotients of the Ore domain $\Z\bar{\Gamma}_n$), we have that:
\begin{equation}
\begin{array}{r}
\text{rk}_{\K_n} H_1(X; R_n)=\text{rk}_{\K_n} H_1(N;R_n)+\sum\limits_{i=1}^r \text{rk}_{\K_n} H_1(S^3_i\backslash K_i;R_n)-\sum\limits_{i=1}^r \text{rk}_{\K_n} H_1(T_i;R_n)+\\
+\text{rk}_{\K_n} \Im(\alpha)+\sum\limits_{i=1}^r \text{rk}_{\K_n} H_0(T_i;R_n)-\text{rk}_{\K_n} H_0(N;R_n)-\\
-\sum\limits_{i=1}^r \text{rk}_{\K_n} H_0(S^3_i\backslash K_i;R_n)+\text{rk}_{\K_n} H_0(X; R_n).
\end{array}
\label{eqnMV}
\end{equation}


Abusing notation, for any $i=1,\ldots,r$ we denote by $\psi:\pi_1(S^3_i\backslash K_i)\longrightarrow \Z$ the (local) linking number homomorphism induced by $\psi:\pi_1(\C^2\backslash C)\longrightarrow \Z$. Then the infinite cyclic cover of $S^3_i\backslash K_i$ induced by the homomorphism $\psi$ is homeomorphic to $F_i\times \R$, where $F_i$ is the Milnor fiber corresponding to the singular point $P_i$. The $\Gamma_n$-cover of $S^3_i\backslash K_i$ factors through this infinite cyclic cover, so we have the following isomorphism of $\K_n$-modules (e.g., see \cite[Section 2.1]{kirk})
\begin{equation}\label{eq3}
H_j(S^3_i\backslash K_i;R_n)\cong H_j(F_i;\K_n), \text{\ \ \ for all }j\geq 0.
\end{equation}
The Milnor fiber $F_i$ has the homotopy type of a wedge sum of $\mu_i$ circles, where $\mu_i$ is the Milnor number associated to the singular point $P_i$. Together with (\ref{eq3}), this yields that $H_2(S^3_i\backslash K_i;R_n)=0.$ Moreover, since the singularity $P_i$ consists of the intersection of 
$d_i$ lines, one has  
$$
\mu_i=(d_i-1)^2
$$
and hence, since the Euler characteristic with coefficients on a $1$-dimensional local system over a skew field does not depend on the local system, we have that
\begin{equation}
\begin{array}{c}
\text{rk}_{\K_n} H_2(S^3_i\backslash K_i;R_n)=0\\
\text{rk}_{\K_n} H_0(S^3_i\backslash K_i;R_n)-\text{rk}_{\K_n} H_1(S^3_i\backslash K_i;R_n)=\chi(F_i)=1-(d_i-1)^2,\\
\end{array}
\label{eqnlink}
\end{equation}
for all $i=1,\ldots, r$.

Similarly, abusing notation again, we denote by $\psi:\pi_1(N)\longrightarrow \Z$ the homomorphism induced by the linking number homomorphism $\psi:\pi_1(\C^2\backslash C)\longrightarrow \Z$. Recall that $N=F\times S^1$, and $F$ is homotopy equivalent to a wedge sum of $r$ circles. From this, we  see that the infinite cyclic cover of $N$ associated to $\psi$ is homeomorphic to $F\times \R$, so it is homotopy equivalent to $F$. Since the $\Gamma_n$-cover of $N$ factors through this infinite cyclic cover, we get as before that
$$
H_j(N;R_n)\cong H_j(F;\K_n), \text{\ \ \ for all }j\geq 0, 
$$
and hence, we have that
\begin{equation}
\begin{array}{c}
\text{rk}_{\K_n} H_2(N;R_n)=0,\\
\text{rk}_{\K_n} H_0(N;R_n)-\text{rk}_{\K_n} H_1(N;R_n)=\chi(F)=1-r.\\
\end{array}
\label{eqnN1}
\end{equation}

Similarly,  the $\Gamma_n$-cover of the torus $T_i$ factors through the infinite cyclic cover of $T_i$ corresponding to the homomorphism induced by the linking number homomorphism $\psi$, and this infinite cyclic cover is homeomorphic to $S^1\times \R$, hence homotopy equivalent to $S^1$. Consequently, we have that
\begin{equation}
H_j(T_i;R_n)\cong H_j(S^1;\K_n), \text{\ \ \ for all }j\geq 0, i=1,\ldots,r, 
\label{eqnS1}
\end{equation}
and hence
\begin{equation}
\text{rk}_{\K_n} H_0(T_i;R_n)-\text{rk}_{\K_n} H_1(T_i;R_n)=\chi(S^1)=0
\label{eqnT}
\end{equation}
for all $i=1,\ldots,r$.

Note that the above calculation (more precisely, the vanishing of the second homology of $N$ and $S^3_i\backslash K_i$) also implies that that the map $\alpha$ in (\ref{eqnMVseq}) is injective. Thus,
$$
\text{rk}_{\K_n}\Im(\alpha)=\text{rk}_{\K_n} H_2(X; R_n).
$$
Since $X$ has the homotopy type of a $2$-dimensional CW-complex (this can be seen from the way $X$ is constructed), we have that $H_2(X; R_n)$ is a free (right) $R_n$-module. Thus, $\text{rk}_{\K_n} H_2(X; R_n)$ is either $0$ or infinite. But
$$
\text{rk}_{\K_n}\Im(\alpha) \leq \sum\limits_{i=1}^r\text{rk}_{\K_n}H_1(T_i;R_n),
$$
and the right hand side of this inequality is a finite number (by (\ref{eqnS1})). Thus,
\begin{equation}
\text{rk}_{\K_n}\Im(\alpha)=0.
\label{eqnker}
\end{equation}

Finally, we show that
\begin{equation}
\text{rk}_{\K_n}H_0(X;R_n)=0
\label{eqnZero}
\end{equation}
by using Fox Calculus (e.g., see \cite[Section 6]{Harvey}). Since $\cA$ is an essential line arrangement, we have that $m\geq 2$. Let $\gamma_1, \ldots, \gamma_m\in \pi_1(X)$ be positively oriented meridians around $L_1,\ldots, L_m,$ respectively.
We fix a presentation of $\pi_1(X)$ with $\{\gamma_1, \ldots, \gamma_m\}$ as the first $m$ generators. The complex of right $R_n$-modules that computes $H_1(X;R_n)$ using this fixed presentation is
$$
\cdots\xrightarrow{\partial_2} (R_n)^l\xrightarrow{\partial_1} R_n\xrightarrow{\partial_0} 0,
$$
where $l\geq m$, and $\partial_1$ is given by the row matrix $\overline{A}$, with
$$
A=\left(\begin{array}{ccccc}
\gamma_1-1\ \ \   & \gamma_2-1\ \ \   & \cdots\ \ \   & \gamma_m-1\ \ \ & \cdots\ \ \  
\end{array}\right).
$$
Here $\overline{A}$ denotes the matrix obtained from $A$ by taking the involution $\overline{\ \cdot\ }$ of all of its entries, and the involution in $\Z[\Gamma_n]$ is given by
$$
\overline{\sum_{\lambda} n_\lambda g_\lambda}=\sum\limits_{\lambda} n_\lambda g_{\lambda}^{-1}.
$$
(The involution is needed here since we are dealing with a complex of right $R_n$-modules, as opposed to the usual formulation of Fox Calculus, where one works with left modules.)
Hence, 
$$
\overline A=\left(\begin{array}{ccccc}
\gamma_1^{-1}-1\ \ \   & \gamma_2^{-1}-1\ \ \   & \cdots\ \ \   & \gamma_m^{-1}-1\ \ \ & \cdots\ \ \
\end{array}\right) .
$$
Let $e_1,\ldots e_l$ be the canonical basis in $(R_n)^l$. We have that
$$
\partial_1((e_1-e_2)\gamma_1)=1-\gamma_2^{-1}\gamma_1,
$$
which is a unit in $R_n$ for all $n$, since $\gamma_2^{-1}\gamma_1\in \bar\Gamma_n$ corresponds to a non-zero element in $\Gamma_0=H_1(\C^2\setminus C;\Z)$. Hence $\partial_1$ is surjective, so
$$
\text{rk}_{\K_n}H_0(X;R_n)=0,
$$
as desired.

Substituting  (\ref{eqnlink}), (\ref{eqnN1}), (\ref{eqnT}), (\ref{eqnker}) and (\ref{eqnZero}) in equation (\ref{eqnMV}), we get that
\begin{equation}\label{eqnTube}
\text{rk}_{\K_n}H_1(X;R_n)=\sum\limits_{i=1}^r\left((d_i-1)^2-1\right)+r-1=\sum\limits_{i=1}^r(d_i-1)^2-1.
\end{equation}

\medskip

The next step in our proof is to relate $\text{rk}_{\K_n}H_1(X;R_n)$ to $\delta_n(C)$. 
Since $L_1$ is not parallel to any other line in the arrangement $\cA$,  the inclusion map $X \hookrightarrow \C^2 \setminus C$ induces an epimorphism
$$
\pi_1(X)\longtwoheadrightarrow \pi_1(\C^2\setminus C).
$$
This can be seen as follows. Let $T$ be a tubular neighborhood of $L_1$ in $\C^2$ such that $X$ is a deformation retract of $T\backslash (T\cap C)$. 
Since $L_1$ is not parallel to any other line in the arrangement, there exists a generic line $L$ (a line transversal to every other line in the arrangement) such that all of the intersections with lines in the arrangement happen in the interior of $T$. Thus, the map
$$
\pi_1((\C^2\backslash C)\cap L\cap T)\longrightarrow \pi_1((\C^2\backslash C)\cap L)
$$
induced by inclusion is an epimorphism, as one can see a set of generators of $\pi_1((\C^2\backslash C)\cap L)$ inside of $(\C^2\backslash C)\cap L\cap T$. By a Zariski theorem of Lefschetz type (\cite[Theorem 6.5, Chapter 1]{dimca}), we have that the map induced by inclusion
$$
\pi_1((\C^2\backslash C)\cap L)\longrightarrow\pi_1(\C^2\backslash C)
$$
is an epimorphism. Then the following commutative diagram yields that $\pi_1(X)\longrightarrow \pi_1(\C^2\backslash C)$ is an epimorphism, where all the arrows in the diagram are induced by inclusion maps.

\begin{center}
\begin{tikzcd}
\pi_1(X) \arrow[rd] \arrow[d] &  \\
\pi_1(T\backslash(T\cap C)) \arrow[r] & \pi_1(\C^2\backslash \C)\\
\pi_1((\C^2\backslash C)\cap L\cap T) \arrow[u] \arrow[r, twoheadrightarrow] & \pi_1((\C^2\backslash C)\cap L) \arrow[u, twoheadrightarrow]
\end{tikzcd}
\end{center}

We have thus shown that $\pi_1(X)\longrightarrow \pi_1(\C^2\setminus C)$ is an epimorphism. This implies (as in the proof of \cite[Theorem 4.1]{MaxLeidy}) that there is an $R_n$-module epimorphism
$$H_1(X;R_n) \longtwoheadrightarrow H_1(\C^2\setminus C;R_n),$$
and hence
\begin{equation}\label{eq5}
\delta_n(C)\leq \text{rk}_{\K_n}H_1(X;R_n)= \sum\limits_{i=1}^r(d_i-1)^2-1.
\end{equation}
Since $L_1$ is not parallel to any other line in the arrangement, we have that
$$
\sum\limits_{i=1}^r (d_i-1)=m-1.
$$
Furthermore, 
$$
d_i-1\leq m-1\text{\ \ \ for all }i=1,\ldots,r,
$$
where the equality is only satisfied in the case where $\cA$ consists of $m$ lines going through a single point. Altogether,
\begin{eqnarray*}
  \delta_n(C) &\leq& \sum\limits_{i=1}^r(d_i-1)^2-1
\leq (m-1)\cdot \left(\sum\limits_{i=1}^r(d_i-1)\right)-1\\
&=& (m-1)^2-1=m(m-2),
\end{eqnarray*}
where the second inequality can only be an equality in the case where $\cA$ consists of $m$ lines going through a single point. In fact, if $\cA$ consists of $m$ lines going through a single point, then $X$ is a deformation retract of $\C^2\setminus C$, so in that case the first inequality is also an equality (since (\ref{eq5}) becomes an equality) and $\delta_n(C)=m(m-2)$ for all $n$. (An alternative proof of the fact that $\delta_n(C)=m(m-2)$ in the case when the arrangement consists of $m$ lines passing through a point was given  in \cite{Suky} by using Fox Calculus.)
\end{proof}

\begin{rem}
In concrete examples, one can use (\ref{eq5}) to get a better (combinatorial) upper bound for $\delta_n(C)$. Moreover, if there are several lines in $\cA$ such that no other line in $\cA$ is parallel to them, we can take the tubes around each of those lines to get different bounds for $\delta_n(C)$ similar to (\ref{eq5}), and then take the minimum of all of these bounds.
\end{rem}

\begin{exm}
Consider the line arrangement of $m$ lines given by $m-1$ parallel lines $L_2, \ldots, L_m$ and a line $L_1$ transversal to all of them. In this case, the tube $X$ around $L_1$ is homotopy equivalent to the arrangement complement, so by  (\ref{eqnTube}) we have that
$$
\delta_n(C)=\text{rk}_{\K_n}H_1(X;R_n)=m-2
$$
for all $n$.
\label{exmnearpencil}
\end{exm}

In view of Lemma \ref{l1}, the following result completes the proof of Theorem \ref{th2}.
\begin{lem}\label{l2}
In the notations of Theorem \ref{th2}, assume that for every line in $\cA$ there exists a different line in $\cA$ that is parallel to it. Then,
$$
\delta_n(C)\leq (m-2)(m-1)-1, \text{\ \ for all }n\text{.}
$$
In particular,
$$
\delta_n(C)\leq m(m-2), \text{\ \ for all }n\text{.}
$$
\end{lem}

\begin{proof}
Reordering, we can assume that the lines $L_1,\ldots, L_k$ are all parallel, with $L_j$ not parallel to $L_1$ for all $k+1\leq j \leq m$, and $k\geq 2$.

Let $\overline{L_1}$ be the closure of $L_1$ in $\C P^2$, and let $T$ be a tubular neighborhood of $\overline{L_1}$ in $\C P^2$ with boundary $\partial T$, constructed so that $\partial T \setminus (C\cup L_\infty)$ is the space
$$
X_\infty=N\cup_{\left(\left(\bigsqcup\limits_{i=1}^r T_i\right)\sqcup T_\infty \right)} \left(\left(\bigsqcup_{i=1}^r S^3_i\backslash K_i\right)\sqcup (S^3_\infty\backslash K_\infty)\right)\subset \C^2\setminus C
$$
defined similarly as the space $X$ from the proof of Lemma \ref{l1}. Here, $L_\infty \subset \C P^2$ is the line at infinity, $S^3_\infty$ is a $3$-sphere centered at the intersection point $P_{\infty}$ of $\overline{L_1}$ with the line at infinity, $K_\infty$ is the link of $P_{\infty}$, and $T_\infty$ is the torus along which we glue $N$ to $S^3_\infty\backslash K_\infty$.

By construction,  $X_\infty$ is a deformation retract of $T\backslash\left(C\cup L_\infty\right)$, and by a similar argument using a Zariski theorem of Lefschetz type (like in the proof of Lemma \ref{l1}), we get that the inclusion map $X_{\infty} \hookrightarrow \C^2\backslash C$ induces an epimorphism
$$
\pi_1(X_\infty)\longtwoheadrightarrow \pi_1(\C^2\setminus C),
$$
which in turn implies that
\begin{equation}
\delta_n(C)\leq \text{rk}_{\K_n} H_1(X_\infty; R_n).
\label{eqndom}
\end{equation}

To compute $\text{rk}_{\K_n} H_1(X_\infty; R_n)$, we follow the same steps as in the proof of Lemma \ref{l1}, based on a Mayer-Vietoris argument. The only difference will appear when computing $H_j(S^3_\infty\backslash K_\infty;R_n)$ for $j=0,1,2$, since the linking number homomorphism $\psi$ satisfies that $\psi(\gamma_\infty)=-m$, where $\gamma_\infty$ is a positively oriented meridian around the line at infinity.

Following the same computation as in the proof of  (\ref{eqnZero}), we get that
\begin{equation}
H_0(S^3_\infty\backslash K_\infty;R_n)=0
\label{eqnZeroParallel}
\end{equation}
To compute $\text{rk}_{\K_n}H_2(S^3_\infty \backslash K_\infty;R_n)$ and $\text{rk}_{\K_n}H_1(S^3_\infty\backslash K_\infty;R_n)$, we will use Fox Calculus, since we cannot relate these groups to the homology of a Milnor fiber. For this, we first need to find a nice presentation of $\pi_1(S^3_\infty\backslash K_\infty)$.

By the choices made in the first paragraph of our proof, $K_\infty$ is the Hopf link on $k+1$ components, with $k\geq 2$. A presentation of $\pi_1(S^3_\infty\backslash K_\infty)$ is given by (e.g., see \cite[Lemma 2.7]{maxtommy})
\begin{eqnarray*}
\pi_1(S^3_\infty\backslash K_\infty)&=&\langle \gamma_1,\gamma_2,\ldots,\gamma_{k},y \mid \gamma_iy\gamma_i^{-1}y^{-1} \text{\ for all }i=1,\ldots,k\rangle, \end{eqnarray*}
where $\gamma_1,\gamma_2,\ldots,\gamma_{k}$ are positively oriented meridians around $L_1,\ldots,L_k$ respectively.
An equivalent presentation of $\pi_1(S^3_\infty\backslash K_\infty)$ can be given so that $y$ is the product of meridian loops $\gamma_1,\gamma_2,\ldots,\gamma_{k}$, $\gamma_{\infty}$ (in a certain order that is not important here), e.g., see \cite[Remark 2.8]{maxtommy}.
 In particular, from this second presentation we get that 
 $$\psi(y)=k-m.$$
 For simplicity, let $a_1=\gamma_1$, and $a_j=\gamma_j\gamma_1^{-1}$ for all $j=2,\ldots, k$. Then, we get
\begin{equation}
\pi_1(S^3_\infty\backslash K_\infty)=\langle a_1,a_2,\ldots,a_{k},y\ |\ a_iya_i^{-1}y^{-1} \text{\ for all }i=1,\ldots,k\rangle,
\label{eqnpresentation}
\end{equation}
with
$$
\begin{array}{l}
\psi(a_1)=1,\\
\psi(a_j)=0, \text{ for all }j=2,\ldots,k,\\
\psi(y)=k-m.\\
\end{array}
$$

We compute $H_1(S^3_\infty\backslash K_\infty;R_n)$ and $H_2(S^3_\infty\backslash K_\infty;R_n)$ as right $R_n$-modules, by using the presentation of $\pi_1(S^3_\infty\backslash K_\infty)$ given in (\ref{eqnpresentation}). The chain complex computing these groups looks like
$$
\ldots\rightarrow (R_n)^{k}\xrightarrow{\partial_2} (R_n)^{k+1}\xrightarrow{\partial_1}R_n\rightarrow 0,
$$
where $\partial_2$ is given by the matrix
$$
\overline{\left(
\begin{array}{ccccc}
1-y & 0 & \ldots & 0\\
0 & 1-y &\ldots & 0 \\
\vdots& \vdots & \ddots & \vdots \\
0& 0& \cdots & 1-y \\
a_1-1 &a_2-1&\cdots& a_{k}-1\\
\end{array}
\right)}=
$$
$$
=\left(
\begin{array}{cccc}
1-y^{-1}& 0 & \ldots & 0\\
0 & 1-y^{-1} & \ldots & 0\\
\vdots & \vdots &\ddots &\vdots \\
0 & 0 & \ldots & 1-y^{-1}\\
a_1^{-1}-1 & a_2^{-1}-1 & \ldots & a_{k}^{-1}-1
\end{array}
\right)
$$
Note that $1-a_j^{-1}$ is not zero in $\Z\bar\Gamma_n$, since $a_j^{-1}$ is not the identity in $\Gamma_0$, for $j=2,\ldots k$. Note also that $y$ commutes with $a_1,\ldots,a_k$ in $\pi_1(S^3_\infty\backslash K_\infty)$. Multiply the $k$-th row by $1-a_{k}^{-1}$ on the left (this is a unit in $R_n$). Add the first row times $1-a_1^{-1}$, the second row times $1-a_2^{-1}$, $\cdots$, the $(k-1)$-st row times $1-a_{k-1}^{-1}$, and the last row times $1-y^{-1}$ to the $k$-th row (all the multiplications are on the left) to get
$$
\left(
\begin{array}{ccccc}
1-y^{-1}& 0 & \ldots & 0& 0\\
0 & 1-y^{-1} & \ldots & 0& 0\\
\vdots & \vdots &\ddots &\vdots &\vdots\\
0& 0& \ldots & 1-y^{-1} & 0\\
0 & 0 & \ldots & 0& 0\\
a_1^{-1}-1 & a_2^{-1}-1 & \ldots & a_{k-1}^{-1}-1&a_{k}^{-1}-1\\
\end{array}
\right)
$$
Note that $a_{k}^{-1}-1$ is a unit in $R_n$, and we multiply the last column by $(a_{k}^{-1}-1)^{-1}$ on the right. Add the last column times $1-a_j^{-1}$ to the $j$-th column for all $j=1,\ldots, k-1$. We get
$$
\left(
\begin{array}{ccccc}
1-y^{-1}& 0 & \ldots & 0& 0\\
0 & 1-y^{-1} & \ldots & 0& 0\\
\vdots & \vdots &\ddots &\vdots &\vdots\\
0& 0& \ldots & 1-y^{-1} & 0\\
0 & 0 & \ldots & 0& 0\\
0 & 0 & \ldots & 0&1\\
\end{array}
\right)
$$
This matrix corresponds to $\partial_2$ after a change of basis in both $(R_n)^{k}$ and $(R_n)^{k+1}$. Therefore, we get $H_2(S^3_\infty\backslash K_\infty;R_n)=0$ and
$$
H_1(S^3_\infty\backslash K_\infty,x_0;R_n)=R_n\oplus \left( R_n/(1-y^{-1}) \right)^{\oplus (k-1)}, 
$$
where $x_0$ is a point in $S^3_\infty\backslash K_\infty$, and there are $k-1$ direct summands of the form $R_n/(1-y^{-1})$. By \cite[Proposition 5.6]{Harvey}, we get that
$$
H_1(S^3_\infty\backslash K_\infty;R_n)=\bigoplus\limits_{k-1\text{ copies}} R_n/(1-y^{-1}),
$$
so
$$
\text{rk}_{\K_n}H_1(S^3_\infty\backslash K_\infty;R_n)=(k-1)(m-k).
$$

Now, as in the proof of Lemma \ref{l1}, we get that
\begin{equation}\begin{split}
\text{rk}_{\K_n}H_1(X_\infty;R_n) &=\sum\limits_{i=1}^{r}\left((d_i-1)^2-1\right)+(k-1)(m-k)+r-1\\
&=\sum\limits_{i=1}^{r}(d_i-1)^2+(k-1)(m-k)-1.
\label{eqnTubebad}
\end{split}
\end{equation}
Since $k\geq 2$, we have that $m-k\leq m-2$. Also, $d_i\leq m-k+1$ for all $i=1,\ldots,r$, and $\sum\limits_{i=1}^{r}(d_i-1)=m-k$. 
Thus,
\begin{eqnarray*}
\sum\limits_{i=1}^{r}(d_i-1)^2+(k-1)(m-k)-1 &\leq& (m-k)\cdot \left(\sum\limits_{i=1}^{r}(d_i-1)+k-1\right)-1\\ &=& (m-k)(m-1)-1 \\ &\leq& (m-2)(m-1)-1,
\end{eqnarray*}
which completes the proof.
\end{proof}

\begin{rem}
The higher order degrees $\delta_n(C)$ are {\it not}  homeomorphism invariants of $\C^2 \setminus C$. For example, if $m\geq 3$,  the case discussed in \cref{exmnearpencil} and that of $m$ lines going through a point have homeomorphic complements. However, as we have discussed, their higher order degrees $\delta_n(C)$ are $m-2$ and $m(m-2)$, respectively, for all $n\geq 0$. This is due to the dependence of higher order degrees on the local system given by the linking number homomorphism.
\end{rem}


\subsection{Vanishing of higher-order degrees} \label{sectionvanish}
In the case when an arrangement  contains one line which meets all other lines transversally, higher order degrees are particularly simple. In this section, we prove the following result, which was asserted (without proof) in \cite[Section 3.1]{MaxLeidySurvey}.

\begin{thm}\label{th4}
Assume that the affine plane curve $C$ defines a line arrangement  $\cA=\{L_1,\ldots, L_m\}\subset \C^2$, which is obtained from an essential line arrangement
 $\cA'=\{L_1,\ldots, L_{m-1}\}$ by adjoining a line $L_m$ that is transversal to every line in $\cA'$ (that is,
 the singularities of the curve $C$ along the irreducible component $L_m$ consist of $m-1$ nodes). Then
$$
\delta_n(C)= 0, \text{ for all }n\geq 0.
$$
\end{thm}

Before proving the theorem, we recall some notation. Let $C'$ be the curve defined by $\cA'$, let $U=\C^2\backslash C$ denote as before the complement of $\cA$, and let $U'=\C^2\setminus C'$ be the complement of $\cA'$. Let $u_0\in U$, which we will take as the base point for the fundamental groups of both $\pi_1(U)$ and $\pi_1(U')$.

By \cite[Lemma 2]{oka}, we have that
\begin{equation}
1\longrightarrow \Z \xrightarrow{g_1} \pi_1(U)\xrightarrow{g_2} \pi_1(U')\longrightarrow 1
\label{eqncentral}
\end{equation}
is a central extension, where $g_1(1)$ is a positively oriented meridian around $L_m$, and the map $g_2$ is induced by inclusion. By the Zariski-Van Kampen  theorem (see, e.g.,  \cite[Chapter 4, Section 3]{dimca}), we find a presentation of $\pi_1(U)$ of the form
\begin{equation}
\pi_1(U)=\langle y_1,\ldots, y_m\mid s_j(y_1,\ldots,y_m)\rangle
\label{eqnU1}
\end{equation}
where $y_1,\ldots, y_{m-1}$ are certain positively oriented meridians about irreducible components of $C'$ and $y_m$ is a positively oriented meridian about $L_m$. All of the $y_i$'s are contained in a generic line section of $\C^2\setminus C$, and $s_j(y_1,\ldots,y_m)$ are certain words on the generators given by braid monodromy. Note that since $y_m$, and $g_1(1)$ are both positively oriented meridians about $L_m$, they must be conjugate, and since $g_1(1)$ is in the center of $\pi_1(U)$, then $y_m=g_1(1)$ in $\pi_1(U)$.

Consider the following splitting of $g_1$
$$
\begin{array}{ccccc}
h: & \pi_1(U) & \longrightarrow &\Z & \\
 & y_i & \mapsto & 0& \text{for } \ i=1,\ldots m-1\\
 & y_m & \mapsto & 1 & \\
\end{array}
$$
which is well defined because it factors through the abelianization of $\pi_1(U)$. Hence, $h(s_j(y_1,\ldots,y_m))=0$ for all $j$, which, along with the fact that $y_m$ is in the center of $\pi_1(U)$, allows us to find an equivalent presentation
\begin{equation}
\pi_1(U)=\langle y_1,\ldots,y_m \mid [y_i,y_m] \text{ for all }i=1,\ldots,m-1; r_j(y_1,\ldots,y_{m-1})\text{ for }j=1,\ldots,l\rangle
\label{eqnU}
\end{equation}
where the $r_j$'s are words in the letters $y_1,\ldots y_{m-1}$. By setting $x_m=y_m$ and $x_i=y_iy_m^{-1}$ for $i=1,\ldots,m-1$, and taking into account that $\psi(r_j(y_1,\ldots,y_{m-1}))=0$ for all $j=1,\ldots, l$, we obtain the following equivalent presentation for $\pi_1(U)$
\begin{equation}
\pi_1(U)=\langle x_1,\ldots,x_m \mid [x_i,x_m] \text{ for all }i=1,\ldots,m-1; r_j(x_1,\ldots,x_{m-1})\text{ for }j=1,\ldots,l\rangle.
\label{eqnpres}
\end{equation}
Using (\ref{eqncentral}) and (\ref{eqnU}), we get the following presentation for $\pi_1(U')$
\begin{equation}
\pi_1(U')=\langle y_1,\ldots,y_{m-1} \mid r_j(y_1,\ldots,y_{m-1})\text{ for }j=1,\ldots,l\rangle.
\label{eqnU'}
\end{equation}

Note that the following map is an isomorphism
$$
\f{f}{\pi_1(U')\times\Z}{\pi_1(U)}{(y_i,t)}{x_i\cdot x_m^t}.
$$
For any $n \geq 0$, we denote by $\Gamma_n(U)$ (resp., $\Gamma_n(U')$) the PTFA group corresponding to $\pi_1(U)$ (resp., $\pi_1(U')$), as in Section \ref{notations}. We then have that $f$ induces an isomorphism
$$
f_n:\Gamma_n(U')\times \Z\longrightarrow\Gamma_n(U).
$$
Moreover, if $\psi:\pi_1(U)\longrightarrow \Z$ is the linking number homomorphism, then $\bar\Gamma_n(U)$ is identified with $\Gamma_n(U')$ via $f_n$, where $$\bar\Gamma_n(U)=\ker\left({\bar \psi}:\Gamma_n(U)\longrightarrow \Z\right),$$
with $\bar \psi$ induced from $\psi$.
As in Section \ref{notations}, we let $S_n=\Z\left[\bar\Gamma_n(U)\right]\backslash \{0\}$,  $\K_n=\Z\left[\bar\Gamma_n(U)\right]S_n^{-1}$, and $R_n=\Z[\Gamma_n(U)]S_n^{-1}$. Let $\cK_n(U')$ denote the (skew) field of quotients of the Ore domain $\Z[\Gamma_n(U')]$.

\begin{rem}
Notice that $f_n$ identifies $\K_n$ with $\cK_n(U')$.
\label{remIdentification}
\end{rem}

Consider the matrix of Fox derivatives for $\pi_1(U')$, that is, 
$$
\left(\frac{\partial r_j(y_1,\ldots, y_{m-1})}{\partial y_i}\right)_{i,j} , \ \ \ 1\leq i \leq m-1, 1\leq j \leq l,
$$
which has entries in $\Z[\pi_1(U')]$, and we take its involution
$$
A=\overline{\left(\frac{\partial r_j(y_1,\ldots, y_{m-1})}{\partial y_i}\right)_{i,j}} .
$$
Let $q_n':\pi_1(U')\longrightarrow \Gamma_n(U')$ be the projection, and let
$$
B(n)=A^{q_n'},
$$
that is, the matrix formed by the images of the entries of $A$ by $q_n'$. Since $\cK_n(U')$ is flat over $\Z[\Gamma_n(U')]$, we have that $B(n)$ is a presentation matrix for the right $\cK_n(U')$-module $H_1(U', u_0;\cK_n(U'))$; again, we refer to \cite[Section 6]{Harvey} for more details about Fox calculus.

\begin{lem}
The rank of the left $\cK_n(U')$-module generated by the rows of $B(n)$ is $m-2$.
\label{lemRank}
\end{lem}

\begin{proof}
Since $U'$ is the complement of an essential line arrangement, we get by Theorem \ref{th2} that $\delta_n(C')$ is finite. By \cite[Remark 3.8]{MaxLeidy}, this means that
$$
\text{rk}_{\cK_n(U')}H_1(U';\cK_n(U'))=0,
$$
and by \cite[Proposition 5.6]{Harvey}, we get that
$$
\text{rk}_{\cK_n(U')}H_1(U',u_0;\cK_n(U'))=\text{rk}_{\cK_n(U')}H_1(U';\cK_n(U'))+1=1.
$$
Since $B(n)$ is an $(m-1)\times l$ matrix, the rank of the left $\cK_n(U')$-module generated by the rows of $B(n)$ (which is the same as the rank of the right $\cK_n(U')$-module generated by the columns of $B(n)$) must be $m-2$.
\end{proof}

We are now ready to prove Theorem \ref{th4}.

\begin{proof}[Proof of Theorem \ref{th4}]
Let $n\geq 0$. We start by considering the presentation matrix for $H_1(U,u_0;R_n)$ as a right $R_n$-module given by the involution of the matrix of Fox derivatives corresponding to the presentation of $\pi_1(U)$ from (\ref{eqnpres}), which is
\begin{equation}
\left(\begin{array}{cccc|ccc}
1-x_m^{-1}& 0 & \cdots & 0 \\
0 & 1-x_m^{-1} & \cdots & 0 & &B(n)\\
\vdots & \vdots & \ddots & \vdots\\
0 & 0 & \cdots & 1-x_m^{-1}\\
\hline
x_1^{-1}-1 & x_2^{-1}-1 & \cdots & x_{m-1}^{-1}-1 & 0 & \cdots & 0\\
\end{array}\right)
\label{eqnmatrix}
\end{equation}
where $B(n)$ is seen as a matrix in $\K_n\subset R_n$ by the identification of $\K_n$ and $\cK_n(U')$ given by $f_n$, and the rest of the entries are seen in $R_n$. By \cref{lemRank}, the rank of the left $\K_n$-module spanned by the rows of $B(n)$ is $m-2$. We denote this by $\text{rk}_{\K_n}B(n)=m-2$.

Note that $(1-x_j^{-1})$ are non-zero elements of $S_n$ for all $j=1,\ldots, m-1$. We multiply the first row (on the left) by $(1-x_1^{-1})$, and then add the $j$-th row times $(1-x_j^{-1})$ to the first row  for all $j=2, \ldots, m$. Taking into account that $x_m$ commutes with everything else, we get
$$
\left(\begin{array}{cccc|ccc}
0& 0 & \cdots & 0 \\
0 & 1-x_m^{-1} & \cdots & 0 & &B_1(n)\\
\vdots & \vdots & \ddots & \vdots\\
0 & 0 & \cdots & 1-x_m^{-1}\\
\hline
x_1^{-1}-1 & x_2^{-1}-1 & \cdots & x_{m-1}^{-1}-1 & 0 & \cdots & 0\\
\end{array}\right)
$$
where $\text{rk}_{\K_n} B_1(n)=m-2$, since we just did row operations in $\K_n$ to get from $B(n)$ to $B_1(n)$. Multiplying the first column (on the right) by $(x_1^{-1}-1)^{-1}$ and then doing column operations, we get
\begin{equation}
\left(\begin{array}{cccc|ccc}
0& 0 & \cdots & 0 \\
0 & 1-x_m^{-1} & \cdots & 0 & &B_1(n)\\
\vdots & \vdots & \ddots & \vdots\\
0 & 0 & \cdots & 1-x_m^{-1}\\
\hline
1 & 0 & \cdots & 0 & 0 & \cdots & 0\\
\end{array}\right)
\label{eqnmatrix2}
\end{equation}

Performing row and column operations in $\K_n$, and using that $\text{rk}_{\K_n} B_1(n)=m-2$, we get a matrix of the form
$$
\left(\begin{array}{cccc|ccc|ccc}
0& * & \cdots & * & 0 & \ldots & 0 &0 &\cdots & 0\\
\hline
0 & * & \cdots & * & & &&0 &\cdots & 0\\
\vdots & \vdots & \ddots & \vdots& &I_{m-2}  & &\vdots &\ddots & \vdots\\
0 & * & \cdots & *& &&&0 &\cdots & 0\\
\hline
1 & 0 & \cdots & 0 & 0 & \cdots & 0&0 &\cdots & 0\\
\end{array}\right)
$$
where $I_{m-2}$ is the identity matrix of dimension $m-2$. Performing column operations we can get this matrix to look like
$$
\left(\begin{array}{cccc|ccc|ccc}
0& * & \cdots & * & 0 & \ldots & 0&0 &\cdots & 0\\
\hline
0 & 0 & \cdots & 0 & & &&0 &\cdots & 0\\
\vdots & \vdots & \ddots & \vdots& &I_{m-2}& &\vdots &\ddots & \vdots\\
0 & 0 & \cdots & 0& & &&0 &\cdots & 0\\
\hline
1 & 0 & \cdots & 0 & 0 & \cdots & 0& 0 & \cdots & 0\\
\end{array}\right)
$$
Permuting the first and last rows, and putting the columns corresponding to $I_{m-2}$ as columns $2,3,\ldots m-1$, we get
$$
\left(\begin{array}{ccc|ccc|ccc}
 &  & & 0 & \cdots & 0&0 &\cdots & 0\\
 & I_{m-1} & &\vdots & \ddots & \vdots&\vdots &\ddots & \vdots\\
 &  & & 0 & \cdots & 0&0 &\cdots & 0\\ 
\hline
0 & \cdots & 0 & * & \cdots & *&0 &\cdots & 0\\
\end{array}\right)
$$
Let $\cK_n$ be the skew field of quotients of $\Z[\Gamma_n(U)]$. By a similar argument as in the proof of \cref{lemRank}, the rank of the left $\cK_n$-module spanned by the rows of this matrix must be $m-1$, so the last row is actually identically $0$. Hence, $\delta_n(C)=0$.
\end{proof}

\section{Plane curves}\label{singinf}

In this section, we adapt some of the results of Section \ref{cla} to the context of affine plane curve complements.

\subsection{Upper bounds on higher-order degrees}
The goal of this section is to prove the following generalization of Theorem \ref{th1}, in which the plane curve $C$ is allowed to have {\it mild} singularities at infinity.

\begin{thm}\label{th3}\label{thmboundtan}
Let $C\subset \C^2$ be a reduced plane curve of degree $m$, let $\overline C$ be its closure in $\C P^2$ and let $L_\infty$ be the line at infinity. Suppose that the intersections of $L_\infty$ and $\overline C$ are either transversal or $L_\infty$ is the tangent line to $\overline C$ at a smooth point and it is a simple tangent there (i.e., it has multiplicity $2$). If any of the following two conditions hold
\begin{enumerate}
\item[(a)] $m=2$;
\item[(b)] at least one of the intersections of $\overline C$ and $L_\infty$ is transversal,
\end{enumerate}
then 
$$
\delta_n(C)\leq m(m-2),
$$ for all $n \geq 0$.
\end{thm}

\begin{proof} Assume that the intersection $\overline C\cap L_\infty$ consists of $r$ distinct points. 
The case $r=m$ corresponds to the curve $C$ being in general position at infinity, which was already considered in Theorem \ref{th1}.
So, without any loss of generality, we may assume that $r \leq m-1$. 
The proof of the theorem in this case relies on a Mayer-Vietoris argument similar to the one we used in the proofs of Lemma \ref{l1} and Lemma \ref{l2}.

Let $T$ be a tube in $\C P^2$ around $L_\infty$, and let $\{P_1,\ldots, P_r\}=\overline C\cap L_\infty$. If $m>2$, we can assume (after reordering) that the intersection of $\overline C$ and $L_\infty$ is transversal at $P_{r-1}$. Then $T \setminus (\overline C\cup L_\infty)$ deformation retracts to the space
$$
X=\partial T\backslash C=N\cup_{\left(\bigsqcup\limits_{i=1}^r T_i\right)} \left(\bigsqcup_{i=1}^r S^3_i\backslash K_i\right)\subset \C^2\setminus C
$$
where, as in the proof of Lemma \ref{l1}, $N$ is the boundary of a tube around the non-singular part of $L_\infty$, that is, a tube around $L_\infty$ minus a disk around every point $P_i$, $S^3_i\backslash K_i$ is the link complement of the singularity of $\overline C\cup L_\infty$ at $P_i$, and $T_i$ is a $2$-torus described as in the proof of Lemma \ref{l1}. Note that $X$ is the link (complement) at infinity, which is the space used in \cite{MaxLeidy} for proving Theorem \ref{th1}, with the difference that if $C$ is in general position at infinity, $X$ is just the complement of the Hopf link on $m$ components.

By a Zariski-Lefschetz type theorem again, the inclusion $X \hookrightarrow \C^2 \setminus C$ induces an epimorphism
$$
\pi_1(X)\longtwoheadrightarrow \pi_1(\C^2\setminus C),
$$
which in turn implies that
$$
\delta_n(C)\leq \text{rk}_{\K_n} H_1(X;R_n).
$$
It thus suffices to show that $\text{rk}_{\K_n} H_1(X;R_n)\leq m(m-2)$. 

The Mayer-Vietoris sequence for the homology of $X$ with $R_n$-coefficients yields the same equality as in formula (\ref{eqnMV}) of Lemma \ref{l1}, so it remains to compute (or bound) all of the terms on the right-hand side of (\ref{eqnMV}).

\medskip

We begin by noticing that $N$ is homotopy equivalent to the cartesian product of a wedge sum of $r-1$ circles and $S^1$ if $r>1$, and to $S^1$ if $r=1$ (and $m=2$). If $r=1$, a direct Fox Calculus computation yields that
\begin{equation}
\begin{array}{c}
H_2(N;R_n)=0, \\
H_1(N;R_n)=0, \\
\text{rk}_{\K_n}H_0(N;R_n)=m,
\end{array}
\label{eqnNr1}
\end{equation}
where one only uses the fact that a positively oriented meridian $\gamma_\infty$ around $L_\infty$ generates $\pi_1(N)\cong \Z$ and $\psi(\gamma_\infty)=-m$.

If $r>1$, a presentation for the fundamental group of $N\simeq \left(\bigvee\limits_{r-1}S^1\right)\times S^1$ is given as:
$$
\pi_1(N)=\langle a_1,\ldots, a_{r-1}, b \ \vert\ [a_i,b] \text{ for }i=1,\ldots,r-1\rangle,
$$
where each $a_i$ corresponds to a circle in the wedge sum, which in turn corresponds to the boundary of a disk centered at $P_i$, while $b$ is a positively oriented meridian about $L_\infty$. In particular, since the intersection of $\overline C$ and $L_\infty$ is transversal at $P_{r-1}$, the loop $a_{r-1}$ can be chosen to be an oriented meridian about the irreducible component of $\overline C$ going through $P_{r-1}$. Hence, $\psi(a_{r-1})=1$, where $\psi$ denotes as before the linking number homomorphism. Setting $x_{r-1}=a_{r-1}$ and $x_j=a_ja_{r-1}^{-\psi(a_j)}$ for all $j=1,\ldots, r-2$, we get the equivalent  presentation
$$
\pi_1(N)=\langle x_1,\ldots, x_{r-1}, b\ \vert\ [x_i,b] \text{ for }i=1,\ldots,r-1\rangle.
$$
The involution of the matrix of Fox derivatives looks like the left-hand side of the matrix in equation (\ref{eqnmatrix}) of \cref{sectionvanish}, after changing $m$ for $r$ and $x_m$ for $b$.

If $r>2$, we have two possible cases: either there exists $j\in\{1,\ldots,r-2\}$ such that $x_j\neq 0$ in $\Gamma_n$ (in which case, by reordering, we can assume that $j=1$), or $x_j=0$ in $\Gamma_n$ for all $j=1,\ldots,r-2$. \\ In the first case, the same computations as in \cref{sectionvanish} yield the left-hand side of the matrix in equation (\ref{eqnmatrix2}) of \cref{sectionvanish}. Using that $\psi(b)=-m$, we get that
$$
\begin{array}{c}
H_2(N;R_n)=0,\\
\text{rk}_{\K_n}H_1(N;R_n)=m(r-2).
\end{array}
$$
We can also see by using Fox Calculus and the fact that $x_1^{-1}-1$ is a unit in $\K_n$  that $\text{rk}_{\K_n}H_0(N;R_n)=0$, so we get that
\begin{equation}
\begin{array}{c}
H_2(N;R_n)=0, \\
\text{rk}_{\K_n}H_1(N;R_n)-\text{rk}_{\K_n}H_0(N;R_n)= m(r-2).
\end{array}
\label{eqnN}
\end{equation}
In fact, these equalities also hold for the case $r=1$ considered in (\ref{eqnNr1}).\\
If $r>2$ and $x_j=0$ in $\Gamma_n$ for all $j=1,\ldots,r-2$, the complex that computes $H_*(N;R_n)$ by Fox Calculus looks like
$$
\longrightarrow (R_n)^{r-1}\xrightarrow{\partial_2} (R_n)^{r}\xrightarrow{\partial_1} R_n\longrightarrow 0
$$
where $\partial_2$ is given by the matrix
$$
\left(\begin{array}{ccccc}
1-b^{-1}& 0 & 0 &\cdots & 0\\
0 & 1-b^{-1} & 0 &\cdots & 0\\
0 & 0 & 1-b^{-1} &\cdots & 0\\
\vdots & \vdots & \vdots & \ddots & \vdots\\
0 & 0 & 0 &\cdots & 1-b^{-1}\\
0 & 0 & \cdots & 0 & x_{r-1}^{-1}-1\\
\end{array}\right)
$$
and $\partial_1$ by
$$
\left(\begin{array}{cccccc}
0\ \ \   & 0\ \ \   & \cdots\ \ \   & 0\ \ \ & x_{r-1}^{-1}-1\ \ \ & b^{-1}-1\ \ \
\end{array}\right) ,
$$
and we can see directly that (\ref{eqnN}) also holds in this case. 

Finally, let us  analyze the case $r=2$. In this case, $N$ is homotopy equivalent to a torus, and the $\Gamma_n$-cover of $N$ factors through the infinite cyclic cover of $N$ corresponding to $\psi$, which is homotopy equivalent to $S^1$ (a similar argument was used in the proof of Lemma \ref{l1}). Thus, in this case,
\begin{eqnarray*}
\text{rk}_{\K_n}H_1(N;R_n)-\text{rk}_{\K_n}H_0(N;R_n)&=& \text{rk}_{\K_n}H_1(S^1;\K_n)-\text{rk}_{\K_n},H_0(S^1;\K_n) \\ &=& \chi(S^1)\\ &=& 0,
\end{eqnarray*}
so the equalities in (\ref{eqnN}) also hold.

\medskip

Let us next compute the local contributions in  (\ref{eqnMV}), i.e., corresponding to the link complements of the $P_i$'s. 
Suppose that the intersection of $\overline C$ and $L_\infty$ is transversal at $P_i$. Then the link $K_i$ of $P_i$ is the Hopf link on $2$ components, and  we can pick meridians $a,b$ around the components of the link to generate $\pi_1(S_i^3\backslash K_i)\cong \Z^2$, and which satisfy
 $\psi(a)=1$ and $\psi(b)=-m$. In this case, $S_i^3\backslash K_i$ deformation retracts onto a torus, so using again an argument involving the infinite cyclic cover, we get that
\begin{equation}
\text{rk}_{\K_n}H_1(S_i^3\backslash K_i;R_n)=\text{rk}_{\K_n}H_0(S_i^3\backslash K_i;R_n),
\label{eqnHopf2}
\end{equation}
where both ranks are finite, and $H_2(S_i^3\backslash K_i;R_n)=0$.

On the other hand, if $L_\infty$ is the tangent line to $\overline C$ at $P_i$, with multiplicity $2$, then $K_i$ is a type $(2,4)$ torus link, which corresponds to the following braid

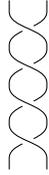
\begin{center}
\begin{tikzpicture}[scale=0.5]
\braid a_1 a_1 a_1 a_1;
\end{tikzpicture}
\end{center}
Again, we can find a presentation for $\pi_1(S_i^3\backslash K_i)$ with two generators and one relation, namely
$$
\pi_1(S_i^3\backslash K_i)=\langle a, b\ \vert\ ababa^{-1}b^{-1}a^{-1}b^{-1}\rangle,
$$
where $a$ corresponds to a positively oriented meridian around the irreducible component of $\overline{C}$ going through $P_i$ and $b$ corresponds to a positively oriented meridian around $L_\infty$. We then have that $\psi(a)=1$ and $\psi(b)=-m$. By setting $x=a$ and $y=ba^m$, we get the equivalent presentation
$$
\pi_1(S_i^3\backslash K_i)=\langle x, y\ \vert\ xyx^{-m+1}yx^{-1}y^{-1}x^{m-1}y^{-1}\rangle.
$$
Let $s=xyx^{-m+1}yx^{-1}y^{-1}x^{m-1}y^{-1}$. When regarded in $\K_n[t^{\pm 1}]$, $\frac{\partial s}{\partial x}$ is a polynomial of degree at most $m-1$, and  $\frac{\partial s}{\partial x}=0$ if only if $y=1$ in $\K_n[t^{\pm 1}]$. Suppose that $y\neq 1$ in $\K_n[t^{\pm 1}]$. Since there is only one relation, a Fox Calculus computation yields that
$$
\text{rk}_{\K_n}H_1(S_i^3\backslash K_i;R_n)\leq m-1,
$$
and
$$
\text{rk}_{\K_n}H_0(S_i^3\backslash K_i;R_n)= 0.
$$
Hence, if $y\neq 1$ in $\K_n[t^{\pm 1}]$, then
\begin{equation}
\text{rk}_{\K_n}H_1(S_i^3\backslash K_i;R_n)-\text{rk}_{\K_n}H_0(S_i^3\backslash K_i;R_n)\leq m-1.
\label{eqnLink}
\end{equation}
If $y= 1$ in $\K_n[t^{\pm 1}]$, we note that $\frac{\partial s}{\partial y}$ is a polynomial of degree $m$, so $\text{rk}_{\K_n}H_1(S_i^3\backslash K_i;R_n)= m$. Moreover, if $\frac{\partial s}{\partial x}=0$, then a Fox Calculus computation yields that $\text{rk}_{\K_n}H_0(S_i^3\backslash K_i;R_n)= 1$. Hence, if $y= 1$ in $\K_n[t^{\pm 1}]$,  the same inequality as in (\ref{eqnLink}) holds. In both cases ($y=1$ and $y\neq 1$ in $\K_n[t^{\pm 1}]$), we see that $H_2(S_i^3\backslash K_i;R_n)=0$.

\medskip

Using a presentation of $\pi_1(X)$ in which there is a generator $a$ such that $\psi(a)=1$ (for example, taking $a$ to be a positively oriented meridian around the irreducible component of $\overline{C}$ going through $P_i$), we also get that
\begin{equation}
\text{rk}_{\K_n}H_0(X;R_n)\leq 1
\label{eqnX0}
\end{equation}
via a Fox Calculus computation.

\medskip

Next, we deal with the contributions to formula (\ref{eqnMV}) of the tori $T_i$, for any $i=1,\ldots, r$. We have the following presentation of the fundamental group
$$
\pi_1(T_i)=\langle c,b \mid [c,b]\rangle,
$$
where $b$ is a positively oriented meridian about $L_\infty$, $\psi(c)=1$ if the intersection of $\overline C$ with $L_\infty$ is transversal at $P_i$, and $\psi(c)=2$ otherwise.

If the intersection of $\overline C$ with $L_\infty$ is transversal at $P_i$, then the $\Gamma_n$-cover of $T_i$ factors through the infinite cyclic cover corresponding to $\psi$. Moreover, since $\psi(c)= 1$, this infinite cyclic cover is homeomorphic to $S^1\times \R$, thus homotopy equivalent to $S^1$. Therefore, 
\begin{equation}\begin{array}{c}
\text{rk}_{\K_n}H_0(T_i;R_n)-\text{rk}_{\K_n}H_1(T_i;R_n)= \text{rk}_{\K_n}H_0(S^1;\K_n)-\text{rk}_{\K_n}H_1(S^1;\K_n)= \chi(S^1)=0, \\
\text{rk}_{\K_n}H_2(T_i;R_n)=\text{rk}_{\K_n}H_2(S^1;\K_n)=0.
\end{array}
\label{eqnTtrans}
\end{equation}
If $L_\infty$ is tangent to $\overline C$ at $P_i$ (note that there are $m-r$ such $P_i$'s), then we have as before that $\psi(c)=2$. There is a $\K_n$-module isomorphism (e.g., see \cite[Section 2.1]{kirk})
$$
H_*(T_i;R_n)\cong H_*(\widetilde {T_i};\K_n),
$$
where $\widetilde {T_i}$ is the (possibly disconnected) infinite cyclic cover of $T_i$ corresponding to $\psi$. Note that $\widetilde {T_i}$ is either homeomorphic to $S^1\times \R$, or to the disjoint union $(S^1\times \R)\sqcup (S^1\times \R)$, depending on whether $m$ is odd or even, respectively. In both cases, $\widetilde {T_i}$ is homotopy equivalent to a one-dimensional finite CW-complex with vanishing Euler characteristic. Hence
\begin{equation}
\begin{array}{c}
\text{rk}_{\K_n}H_0(T_i;R_n)-\text{rk}_{\K_n}H_1(T_i;R_n)=\text{rk}_{\K_n}H_0(\widetilde {T_i};\K_n)-\text{rk}_{\K_n}H_1(\widetilde {T_i};\K_n)=\chi(\widetilde {T_i})=0, \\
\text{rk}_{\K_n}H_2(T_i;R_n)=\text{rk}_{\K_n}H_2(\widetilde {T_i};\K_n)=0.
\end{array}
\label{eqnTtan}
\end{equation}

\medskip

Arguing as in the proof of Lemma \ref{l1}, we also get that
\begin{equation}
\text{rk}_{\K_n}\Im(\alpha)=0.
\label{eqnker2}
\end{equation}

\medskip

Altogether, substituting  (\ref{eqnN}), (\ref{eqnHopf2}), (\ref{eqnLink}), (\ref{eqnX0}), (\ref{eqnTtrans}), (\ref{eqnTtan}) and (\ref{eqnker2}) into (\ref{eqnMV}), we get that
$$
\text{rk}_{\K_n}H_1(X;R_n)\leq m(r-2)+(m-r)(m-1)+1=m^2-3m+r+1.
$$
Moreover, since we assumed that $r\leq m-1$, we get that
$$
\text{rk}_{\K_n}H_1(X;R_n)\leq m^2-2m=m(m-2),
$$
thus concluding the proof.
\end{proof}

\subsection{Vanishing of higher-order degrees}
In the case when an irreducible component of a plane curve $C$ is a line $L$ which meets all other components of $C\cup L_\infty$ transversally (with $L_\infty$ denoting the line at infinity in $\C P^2$), higher order degrees are particularly simple. This is exemplified in the following generalization of \cref{th4}.
\begin{thm}\label{thmVanishing}
Let $n\geq 0$. Assume that the affine plane curve $C$ is of the form $C=L\cup C'$, where $C'$ is a curve of degree $m-1$ in $\C^2$ such that $\delta_n(C')$ is finite, and $L$ is a line transversal to $C'$, such that $L\cap C'$ consists of $m-1$ distinct points. Then,
$$
\delta_n(C)= 0.
$$
\end{thm}

\begin{proof}
The proof follows the same steps as in the proof of \cref{th4}. Let $U=\C^2\setminus C$ and $U'=\C^2\setminus C'$. The Zariski-Van Kampen theorem (see, e.g., \cite[Chapter 4, Section 3]{dimca}) can be used to find a presentation of $\pi_1(U)$ such as the one described in (\ref{eqnU1}). Using \cite[Lemma 2]{oka} and the same arguments as in the proof of \cref{th4}, we find presentations of $\pi_1(U)$ and $\pi_1(U')$ such as the ones described in (\ref{eqnU}) and (\ref{eqnU'}), respectively. Thus, \cref{remIdentification} still holds in this setting. The rest of the proof follows the same Fox Calculus computation as in the proof of \cref{th4}, except for the proof of \cref{lemRank}, which in this case follows from the hypothesis that $\delta_n(C')$ is finite.
\end{proof}

\section{Relationship with the first characteristic variety}\label{char}

In this section, we relate the higher-order degrees to more classical Alexander-type invariants.
We begin by recalling the following result:

\begin{prop}(\cite[Proposition 5.1]{MaxLeidy}) \label{propAlex} \ 
If C is an irreducible affine plane curve, then $$\delta_0(C)=\deg \Delta_C(t),$$
where $\Delta_C(t)$ denotes the Alexander polynomial of the curve complement. If, moreover, the Alexander polynomial is trivial, then all higher-order degrees vanish.
\end{prop}

In what follows, we generalize the above result to non-irreducible affine plane curves. Let us first  introduce some notation, following the conventions from \cite{suciu}.

Let $R$ be a Noetherian commutative ring with unit. Assume also that $R$ is a unique factorization domain. Let $M$ be a finitely generated $R$-module. Then $M$ admits a finite presentation of the form
$$
R^q\overset{\Phi} {\longrightarrow} R^m\longrightarrow M.
$$

\begin{defn}\label{de1}
The $i$-th {\it elementary ideal} of $M$, denoted $E_i(M)$, is the ideal of $R$ generated by the minors of size $m-i$ of the $m\times q$ matrix $\Phi$, with the convention that $E_i(M)=R$ if $i\geq m$ and $E_i(M)=0$ if $m-i>q$. 
\end{defn}

\begin{rem}
The $i$-th elementary ideal does not depend on the choice of representation of $M$ as an $R$-module.
\end{rem}
\begin{rem} It follows immediately from Definition \ref{de1} that 
$$E_i(M)\subset E_{i+1} (M)$$ for all $i\geq 0$.
\end{rem}

\begin{defn}
Let $i\geq 0$. We define $\Delta_i(M)\in R$ to be the generator of the smallest principal ideal in $R$ containing $E_i(M)$, that is, the greatest common divisor of all elements of $E_i(M)$.
\end{defn}

\begin{rem}
$\Delta_i(M)$ is well-defined up to multiplication by a unit of $R$.
\end{rem}

Let $G=\langle x_1,\ldots, x_m \ | \ r_1,\ldots, r_q\rangle$ be a finitely presented group, let $H$ be its maximal, torsion free abelian quotient, and let $\pi:G\longrightarrow H$ be the quotient map. The group ring $\Z H$ is a commutative Noetherian ring with unit, which is also a unique factorization domain.

Let $F_m$ be the free group with generators $x_1,\ldots, x_m$. For each $1\leq j \leq m$, there is a linear operator $$\frac{\partial}{\partial x_j}:\Z F_m \longrightarrow \Z F_m$$
(called the {\it $j$-th Fox derivative}) uniquely determined by the following properties:
\begin{itemize}
\item[(a)] $\frac{\partial 1}{\partial x_j}=0$,
\item[(b)]  $\frac{\partial x_i}{\partial x_j}=\delta_{ij}$,
\item[(c)] $\frac{\partial uv}{\partial x_j}=\frac{\partial u}{\partial x_j}+u\frac{\partial v}{\partial x_j}$, for any $u,v\in F_m.$
\end{itemize}
Let $\phi:F_m\longrightarrow G$ be the presenting homomorphism.

\begin{defn}
The {\it Alexander matrix} of the given presentation of  $G$ is
$$
\Phi_G=\left(\frac{\partial r_i}{\partial x_j}\right)^{\pi\circ\phi}:(\Z H)^q\longrightarrow (\Z H)^m.
$$
\end{defn}

Now, let $X$ be a connected CW-complex with a unique $0$-cell $x_0$, and finitely many $1$-cells. Let $G=\pi_1(X,x_0)$ be the fundamental group, and let $H$ be its maximal torsion-free abelian quotient, that is, $H\cong \Z^{b_1(G)}$. The canonical projection $\pi:G\longrightarrow H$ defines a local system of $\Z H$-modules. The long exact sequence for the homology of the pair $(X,x_0)$ with coefficients in $\Z H$ given by this local system is
$$
\ldots \rightarrow 0 \rightarrow H_1(X;\Z H)\rightarrow H_1(X,x_0; \Z H) \rightarrow H_0(x_0;\Z H)\rightarrow H_0(X; \Z H) \rightarrow 0,
$$
where
$$
\ker\left( H_0(x_0;\Z H)=\Z H\rightarrow H_0(X; \Z H) \right)
$$
can be identified with the augmentation ideal $I_H=\ker (\epsilon: \Z H\rightarrow \Z)$, and $\epsilon$ is defined as
$$
\f{\epsilon}{\Z H}{\Z}{\sum n_i h_i}{\sum n_i}
$$
for $n_i\in \Z$, $h_i\in H$.

The $\Z H$-modules $H_1(X;\Z H)$ and $H_1(X,x_0; \Z H)$ depend only on the fundamental group $G$, so we denote them by $B_G$ and $A_G$, respectively. From the above discussion, these modules fit into the following short exact sequence of $\Z H$-modules
$$
0\rightarrow B_G \rightarrow A_G \rightarrow I_H \rightarrow 0.
$$

\begin{defn}
The $\Z H$-module $B_G$ is called the {\it Alexander invariant} of $X$, and $A_G$ is called the {\it Alexander module} of $X$.
\end{defn}

\begin{defn}\label{defnAP}
The {\it Alexander polynomial} of the group $G$, denoted by $\Delta_G$, is defined by
$$
\Delta_G=\Delta_1(A_G)=\gcd(E_1(A_G))\in \Z H.
$$
\end{defn}

\begin{rem}
Sometimes, the Alexander polynomial of the fundamental group of a CW-complex $X$ is defined as $\Delta_0(B_G)=\gcd(E_0(B_G))$, but the definition using the Alexander module instead of the Alexander invariant is more suitable for our purposes. In the case when $b_1(X)=1$ (e.g., $X=\C^2 \setminus C$, where $C$ is an irreducible affine plane curve), the two definitions coincide. Besides, the Alexander matrix $\Phi_G$ provides a presentation for $A_G$, so we can compute $\Delta_G$ from it.
\label{remAG}
\end{rem}

\begin{defn}
Let $X$ be a connected  CW-complex with finite $k$-skeleton, and let $G=\pi_1(X)$. The {\it homology jump loci} of $X$ (over $\C$) are the Zariski closed sets
$$
\cV_d^i(X)=\{\rho\in \Hom(G,\C^*) \mid \dim_\C H_i(X; \C_\rho)\geq d\},
$$
where $\C_\rho$ is the rank-one $\C$-local system on $X$ induced by $\rho$, $0\leq i \leq k$, and $d>0$. When $i=1$, we use the simplified notation $\cV_d(X)$ for $\cV_d^1(X)$.
\end{defn}

Let $X$ be a connected CW complex with a unique $0$-cell and finitely many $1$-cells. Let $\Hom(G,\C^*)^0$ be the identity component of the algebraic group $\Hom(G,\C^*)$. The projection map $\pi:G\longrightarrow H$ induces an isomorphism $\pi^*:\Hom(H,\C^*)\longrightarrow \Hom(G,\C^*)^0$.



\begin{defn}
The {\it characteristic varieties} of $X$ (over $\C$), denoted by $V_d(X)$, are the subvarieties of $\left(\C^*\right)^{b_1(X)}\cong \Spec \C H$ given by
$$
V_d(X)=V(E_{d-1}(B_G\otimes \C)).
$$
\end{defn}

We recall the following result from \cite{Hironaka}, see also \cite[Proposition 4.7]{suciu}.

\begin{prop} \label{propHironaka} 
Let $\rho:H\longrightarrow \C^*$ be a non-trivial character. Then, for all $d\geq 1$,
$$
\pi^*(\rho)\in\cV_d(X)\Longleftrightarrow \rho\in V\left(E_d(A_G\otimes \C)\right)\Longleftrightarrow \rho\in V_d(X).
$$
\end{prop}

\begin{rem}
If $X=\C^2 \setminus C$ is a plane curve complement, then $\pi^*$ is the identity, and $\Hom(G,\C^*)\cong \Hom(H,\C^*)$ can be identified with $(\C^*)^m$, where $m$ is the number of irreducible components of the plane curve. In this case, \cref{propHironaka} asserts that, away from $(1,\ldots, 1)$, $\cV_d$ and $V_d$ coincide. Moreover, by \cref{remAG}, away from $(1,\ldots, 1)$, we can compute $V_d$ from the dimension $m-d$ minors of the Alexander matrix $\Phi_G$.
\end{rem}

From now on, let $C$ be a plane curve in $\C^2$ with $m$ irreducible components, with $m\geq 2$, and let $G=\pi_1(U, u_0)$, where $U=\C^2\backslash C$. We denote by $\Delta_C\in \Z[t_1^{\pm 1},\ldots, t_m^{\pm 1}]$ the Alexander polynomial $\Delta_G$ of \cref{defnAP}.

\begin{defn}
Let $q$ be a Laurent polynomial in $\C[t_1^{\pm 1}, \ldots, t_m^{\pm 1}]$. We can write $q$ as
$$
q=\sum\limits_k a_kt_1^{i_1^k}\cdot\ldots\cdot t_m^{i_m^k},
$$
for some $a_k\in \C^*$, and $(i_1^k,\ldots,i_m^k)\neq (i_1^j,\ldots,i_m^j)$ for all $k\neq j$. We define the {\it degree} of $q$ as
$$
\deg (q)=\max_k \left(\sum\limits_{l=1}^m i_l^k\right)-\min_j \left(\sum\limits_{l=1}^m i_l^j\right)
$$
\end{defn}

\begin{rem}
The degree of a Laurent polynomial $q$ is $0$ if and only if $q$ is a homogeneous polynomial up to multiplication by a unit of $\C[t_1^{\pm 1}, \ldots, t_m^{\pm 1}]$. Multiplying by a unit of $\C[t_1^{\pm 1}, \ldots, t_m^{\pm 1}]$ does not change the degree, so $\deg(\Delta_C)$ is well-defined.
\end{rem}

\begin{lem}\label{lemIdeal}
$$E_0(H_1(U;R_0))=E_1(H_1(U,u_0;R_0)).$$
\end{lem}
\begin{proof}
Note that $R_0$ is a commutative ring. From the long exact sequence of a pair with coefficients in $R_0$, we get
$$
0\rightarrow H_1(U;R_0)\rightarrow H_1(U,u_0;R_0) \rightarrow H_0(u_0;R_0)\rightarrow H_0(U;R_0)
$$
If $m\geq 2$, following the same Fox Calculus computation as in the proof of Lemma \ref{l1} for $H_0(X;R_0)$, we get that $H_0(U;R_0)=0$, so we have the short exact sequence of $R_0$-modules
$$
0\rightarrow H_1(U;R_0)\rightarrow H_1(U,u_0;R_0) \rightarrow R_0\rightarrow 0
$$
which splits, so the result follows.
\end{proof}

In the above notation, we have the following result relating the multivariate Alexander polynomial and the zero-th higher order degree, thus extending  \cref{propAlex} to the non-irreducible case.


\begin{thm}\label{thmAP}
Suppose $m\geq 2$, and assume that $\delta_0(C)$ is finite. Then,
$$
\delta_0(C)=\deg(\Delta_C).
$$
\end{thm}

\begin{proof}
Note that all the rings that we will deal with in this proof are commutative, so we will not  distinguish between left and right modules. Let $\gamma_1,\ldots, \gamma_m$ be positively oriented meridians around the different components of $C$, let $\psi:\pi_1(U)\rightarrow \Z$ be the linking number homomorphism. Consider the splitting of $\psi$ given by
$$
\f{\phi}{\Z}{\pi_1(U)}{1}{\gamma_1}
$$
which we use to identify $R_0$ with $\K_0[t^{\pm 1}]$. With this identification, we can think of $R_0$ as
$$
Q\left(\Z\left[\left(\frac{t_2}{t_1}\right)^{\pm 1}, \ldots,\left(\frac{t_m}{t_1}\right)^{\pm 1}\right]\right)[t_1^{\pm 1}]
$$
where $\K_0=Q\left(\Z\left[\left(\frac{t_2}{t_1}\right)^{\pm 1}, \ldots,\left(\frac{t_m}{t_1}\right)^{\pm 1}\right]\right)$ is the field of quotients of $\Z\left[\left(\frac{t_2}{t_1}\right)^{\pm 1}, \ldots,\left(\frac{t_m}{t_1}\right)^{\pm 1}\right]$. We can think of $\Z[\Gamma_0]$ as $\Z[t_1^{\pm 1}, \ldots, t_m^{\pm 1}]$ seen inside of $R_0$ this way. Note that the degree of any $q\in\Z[t_1^{\pm 1}, \ldots, t_m^{\pm 1}]$ is the same as the degree of $q$ in $t$ seen as an element of $\K_0[t^{\pm 1}]$.

Let $f_1,\ldots, f_r$ be a set of generators of $E_1(A_G)$. We have that
$$
\Delta_C=\gcd_{\Z[t_1^{\pm 1}, \ldots, t_m^{\pm 1}]}(f_1,\ldots, f_r).
$$
Since $R_0$ is a flat $\Z[\Gamma_0]$-module,  $f_1,\ldots, f_r$ are also a set of generators of $E_1(H_1(U,u_0;R_0))$. By assumption, $\delta_0(C)$ is finite, which means that $E_1(H_1(U,u_0;R_0))$ is not the zero ideal. Moreover, $R_0\cong \K_0[t^{\pm 1}]$ is a PID, so there exists $p\in \K_0[t^{\pm 1}]$ such that
$$
E_1(H_1(U,u_0;R_0))=(p)
$$
and, by definition, $\delta_0(C)=\deg_{\K_0[t^{\pm 1}]} (p)$. Moreover, by considering the prime decomposition of $f_1,\ldots,f_r$ in $\Z[t_1^{\pm 1}, \ldots, t_m^{\pm 1}]$ and applying \cref{lemPrimes} (below), we can choose such $p\in \Z[t_1^{\pm 1}, \ldots, t_m^{\pm 1}]$ such that there exists $a\in \Z[t_1^{\pm 1}, \ldots, t_m^{\pm 1}]$, where $a$ is a unit in $R_0$, with 
$$
p\cdot a=\Delta_C
$$
But units in $R_0$ which are contained in $\Z[t_1^{\pm 1}, \ldots, t_m^{\pm 1}]$ all have degree 0 as elements of $\Z[t_1^{\pm 1}, \ldots, t_m^{\pm 1}]$, therefore
$$
\deg_{\Z[t_1^{\pm 1}, \ldots, t_m^{\pm 1}]}(\Delta_C)=\deg_{\Z[t_1^{\pm 1}, \ldots, t_m^{\pm 1}]}(p)=\delta_0(C).
$$
\end{proof}

\begin{lem}
Let $R=Q\left(\Z\left[\left(\frac{t_2}{t_1}\right)^{\pm 1}, \ldots,\left(\frac{t_m}{t_1}\right)^{\pm 1}\right]\right)[t_1^{\pm 1}]$, and $S=\Z[t_1^{\pm 1}, \ldots, t_m^{\pm 1}]$. Let $q$ be a prime element in $S$. Then, $q$ is either prime or a unit in $R$.
\label{lemPrimes}
\end{lem}

\begin{proof}
This is an exercise in commutative algebra, which is a direct consequence of \cite[Proposition 3.11, iv)]{Atiyah}, for example.
\end{proof}

As a consequence of \cref{propHironaka}, \cref{propAlex}, \cref{thmAP} and \cite[Corollary 3.2]{DPS}, we get the following:
\begin{cor}\label{corcodim}
$$\codim V_1(U)>1\Longleftrightarrow \Delta_C\in\Z\backslash 0 \Longrightarrow \delta_0(C)=0.$$
\end{cor}

Moreover, \cref{propAlex,thmAP} imply the following:
\begin{cor}
$\delta_0(C)=0$ if and only if  $\Delta_C$ has a representative in $\Z[t_1^{\pm 1},\ldots,t_m^{\pm 1}]$ that is a non-zero homogeneous polynomial.
\end{cor}

\begin{rem}
Note that since $\cK_0$ is flat over $\Z[\Gamma_0]$, by \cite[Proposition 5.6]{Harvey} we have that $$\codim V_1(U)=0\Longleftrightarrow E_1(A_G)=(0)\Longleftrightarrow \text{rk}_{\mathcal K_0}H_1(U, u_0;\cK_0)>1\Longleftrightarrow \delta_0 \text{ is not finite}. $$

Hence, if $\delta_0(C)$ is finite (as is the case for all curves in general position at infinity (\cite{MaxLeidy}), the ones described in \cref{thmboundtan}, or all line arrangements except the one consisting of $m$ parallel lines, for $m\geq 2$), then $\codim V_1(U)\geq 1$.
\end{rem}

\begin{exm}\label{exmAP}
Let $C$ be the line arrangement described in \cref{exmnearpencil}, and let $t_i$ the variable corresponding to the component $L_i$, for $i=1,\ldots, m$. Then
$$
\Delta_C=(t_1-1)^{m-2}.
$$
\end{exm}

Let us now specialize our results to the case when $C$ is an essential line arrangement. We recall the following result from \cite{suciu}. 

\begin{prop} (\cite[Theorem 9.15]{suciu}) \label{propDPS} \ 
Let $C$ be an essential line arrangement. Then:
\begin{enumerate}
\item If $C$ consists of $m$ lines going through a point (a pencil of lines), with $m\geq 3$, then $$\Delta_C=(t_1t_2\ldots t_m-1)^{m-2}.$$
\item If $C$ is as in \cref{exmnearpencil}, then $\Delta_C=(t_1-1)^{m-2}$.
\item For all other essential arrangements, $\Delta_C\in \Z\setminus \{ 0\}$.
\end{enumerate}
\end{prop}

Then \cref{thmAP} and \cref{propDPS} have the following consequence:
\begin{cor}
Let $C$ be an essential line arrangement. Then:
\begin{enumerate}
\item $\codim V_1(C)>1\Longleftrightarrow \delta_0(C)=0$. That is, the result from \cref{corcodim} is an ``if and only if'' for line arrangements.
\item If $C$ is a pencil of lines, then $\delta_0(C)=m(m-2)$.
\item If $C$ is as in \cref{exmnearpencil}, then $\delta_0(C)=(m-2)$.
\item If $C$ is not as in items (2) or (3), then $\delta_0(C)=0$.
\end{enumerate}
\label{cordelta0}
\end{cor}

\begin{rem}
By Lemma \ref{l1} and \cref{exmAP}, $\delta_n(C)=\delta_0(C)$ for all $n$ for the cases described in items (2) and (3) of the above Corollary.
\end{rem}


\end{document}